\documentclass{conm-p-l}
\usepackage{amsmath,amssymb,amsfonts,amsthm}
\usepackage[latin1]{inputenc}
\usepackage[OT2,T1]{fontenc}
\usepackage{mathrsfs}
\usepackage{url}
\usepackage{color}

\newcommand{\Aut}{\mathrm{Aut}}

\newcommand{\Hom}{\mathrm{Hom}}

\newcommand{\Mul}{\mathbf{M}}
\newcommand{\Sqr}{\mathbf{S}}
\newcommand{\mul}{\mathbf{m}}

\newcommand{\bbmu}{\mu\!\!\!\!\mu}

\newcommand{\isom}{\cong}

\renewcommand{\AA}{\mathbb{A}}

\newcommand{\PP}{\mathbb{P}}

\newcommand{\ZZ}{\mathbb{Z}}
\newcommand{\cI}{\mathscr{I}}
\newcommand{\cL}{\mathscr{L}}
\newcommand{\cO}{\mathcal{O}}
\newcommand{\ignore}[1]{}

\newcommand{\oO}{O}
\newcommand{\fs}{\mathfrak{s}}
\newcommand{\ccomma}{\raisebox{0.4ex}{,}}

\newtheorem{theorem}{Theorem}[section]
\newtheorem{lemma}[theorem]{Lemma}
\newtheorem{corollary}[theorem]{Corollary}

\theoremstyle{definition}
\newtheorem{definition}[theorem]{Definition}

\begin{document}

\title{The geometry of efficient arithmetic on elliptic curves}
\begin{abstract}
The arithmetic of elliptic curves, namely polynomial addition 
and scalar multiplication, can be described in terms of global 
sections of line bundles on $E\times  E$ and $E$, respectively,
with respect to a given projective embedding of $E$ in $\PP^r$. 
By means of a study of the finite dimensional vector spaces of 
global sections, we reduce the problem of constructing and 
finding efficiently computable polynomial maps defining the 
addition morphism or isogenies to linear algebra.  
We demonstrate the effectiveness of the method by improving 
the best known complexity for doubling and tripling, by 
considering families of elliptic curves admiting a $2$-torsion 
or $3$-torsion point.
\end{abstract}
\author{David Kohel}
\thanks{This work was supported by a project of the Agence Nationale de la Recherche, reference ANR-12-BS01-0010-01.}
\address{Aix Marseille Université, CNRS, Centrale Marseille, I2M, UMR 7373, 13453 Marseille, France}
\curraddr{\newline 
Institut de Mathématiques de Marseille (I2M)\newline
163 avenue de Luminy, Case 907\newline
13288 Marseille, cedex 9\newline France}
\email{David.Kohel@univ-amu.fr}

\maketitle

\section{Introduction}

The computational complexity of arithmetic on an elliptic curve, 
determined by polynomial maps, depends on the choice of projective 
embedding of the curve. Explicit counts of multiplications and 
squarings are expressed in terms of operations on the coordinate 
functions determined by this embedding. The perspective of this work 
is to reduce the determination of the complexity of evaluating a 
morphism, up to additions and multiplication by constants, to a 
problem of computing a $d$-dimensional subspace $V$ of the space 
of monomials of degree~$n$. This in turn can be conceptually 
reduced to the construction of a flag 
$
V_1 \subset V_2 \subset \cdots \subset V_d = V.
$
For this purpose we recall results from Kohel~\cite{Kohel2011} 
for the linear classification of projectively normal models.
We generalize this further by analysing the conditions under which 
a degree~$n$ isogeny is determined by polynomials of degree~$n$
in terms of the given projective embeddings. This approach allows us 
to derive conjecturally optimal or nearly optimal algorithms for 
operations of doubling and tripling, which form the basic building 
blocks for efficient scalar multiplication.

\section{Background}

An elliptic curve $E$ is a projective nonsingular genus one curve 
with a fixed base point.  In order to consider the arithmetic, 
namely addition and scalar multiplication defined by polynomial 
maps, we need to fix the additional structure of a 
projective embedding.  We call an embedding $\iota: E \rightarrow 
\PP^r$ a (projective) model for $E$. A model given by a complete 
linear system is called {\it projectively normal} 
(see Birkenhake-Lange~\cite[Chapter~7, Section~3]{BirkenhakeLange} or 
Hartshorne~\cite[Chapter~I, Exercise~3.18 \& Chapter~II, 
Exercise~5.14]{Hartshorne} for the general definition and 
its equivalence with this one for curves).
If $\iota: E \rightarrow \PP^r$ is a projectively normal model, 
letting $\{ X_0,\dots, X_r \}$ denote the coordinate functions 
on $\PP^r$, we have a surjection of rings:
$$
\iota^* : k[\PP^r] = k[X_0,\dots,X_r] \longrightarrow 
k[E] = \frac{k[X_0,\dots,X_r]}{I_E}\ccomma
$$
where $I_E$ is the defining ideal for the embedding. 
In addition, using the property that $\iota$ is given by a 
complete linear system, there exists $T$ in $E(k)$ such 
that every hyperplane intersects $E$ in $d = r + 1$ points 
$\{ P_0,\dots,P_r \} \subset E(\bar{k})$, with multiplicities, 
such that $P_0 + \cdots P_r = T$, 
and we say that the degree of the embedding is $d$.  
The invertible sheaf $\cL$ attached to the embedding is 
$\iota^* \cO_{\PP^r}(1)$, where $\cO_{\PP^r}(1)$ is the 
sheaf spanned by $\{X_0,\dots,X_r\}$.  
Similarly, the space of global sections of $\cO_{\PP^r}(n)$ 
is generated by the monomials of degree $n$ in the $X_i$. 
Let $\cL^n$ denote its image under $\iota^*$, then the global 
sections $\Gamma(E,\cL^n)$ is the finite dimensional $k$-vector 
space spanned by monomials of degree $n$ modulo $I_E$, and 
hence
$$
k[E] = \bigoplus_{n=0}^\infty \Gamma(E,\cL^n), 
$$
which is a subspace of $k(E)[X_0]$.

Now let $D$ be the divisor on $E$ cut out by $X_0 = 0$, then
we can identify $\Gamma(E,\cL^n)$ with the Riemann--Roch space 
associated to $nD$:
$$
L(nD) = \{ f \in k(E)^* \;|\; \mathrm{div}(f) \ge -nD \} \cup \{0\}.
$$
More precisely, we have $\Gamma(E,\cL^n) = L(nD) X_0^n \subset k(E)X_0^n$ 
for each $n \ge 0$. 
While the dimension of $L(nD)$ is $nd$, the dimension of the space 
of all monomials of degree $n$ is:
$$
\dim_k\big(\Gamma(\PP^r,\cO_{\PP^r}(n))\big) 
  = \binom{n+r}{r} = \binom{n+d-1}{d-1}\cdot
$$
The discrepancy is accounted for by relations of a given degree in 
$I_E$. More precisely, for the ideal sheaf $\cI_E$ of $E$ on $\PP^r$, 
with Serre twist $\cI_E(n) = \cI_E \otimes \cO_{\PP^r}(n)$, the 
space of relations of degree $n$ is $\Gamma(\PP^r,\cI_E(n))$, such 
that the defining ideal of $E$ in $\PP^r$ is
$$ 
I_E = \bigoplus_{n=1}^\infty \Gamma(\PP^r,\cI_E(n)) \subset k[X_0,\dots,X_r].
$$
Consequently, each polynomial in the quotient $\Gamma(E,\cL^n) \subset k[E]$ 
represents a coset $f + \Gamma(\PP^r,\cI_E(n))$ of polynomials. 

From the following table of dimensions:
$$
\begin{array}{r|ccr|cc}
\multicolumn{3}{c}{d = 3:}\\[1mm] 
    n & \binom{n+r}{r} & nd \\[1mm]  \hline 
    1 & 3  &  3    \\
    2 & 6  &  6    \\
    3 & 10 &  9    
\end{array}
\hspace{4mm}
\begin{array}{r|cc}
\multicolumn{3}{c}{d = 4:}\\[1mm] 
    n & \binom{n+r}{r} & nd \\[1mm]  \hline 
    1 &  4 &    4 \\
    2 & 10 &    8 \\ 
    3 & 15 &   12 
\end{array}
\hspace{4mm}
\begin{array}{r|cc}
\multicolumn{3}{c}{d = 5:}\\[1mm] 
    n & \binom{n+r}{r} & nd \\[1mm]  \hline 
    1 &  5 &    5 \\
    2 & 15 &   10 \\ 
    3 & 35 &   15 
\end{array}
\hspace{4mm}
\begin{array}{r|cc}
    \multicolumn{3}{c}{d = 6:}\\[1mm] 
    n & \binom{n+r}{r} & nd \\[1mm]  \hline 
    1 &  6 &    6 \\
    2 & 21 &   12 \\ 
    3 & 56 &   18 
\end{array}
$$
we see the well-known result that a degree-3 curve 
in $\PP^2$ is generated by a cubic relation, and a 
degree-4 curve in $\PP^3$ is the intersection of 
two quadrics. Similarly, a quintic model in $\PP^4$
and a sextic model in $\PP^5$ are generated by a 
space of quadrics of dimensions 5 and 9, respectively.

When considering polynomial maps between curves, this 
space of relations $\Gamma(\PP^r,\cI_E(n))$, which 
evaluate to zero, gives a source of ambiguity but also 
room for optimization when evaluating a representative 
polynomial $f$ in its class. %$f + \Gamma(\PP^r,\cI_E(n))$. 

\subsection*{Addition law relations}

A similar analysis applies to the set of addition laws,
from $E \times E$ to $E$.
The set polynomials of bidegree $(m,n)$ on $E \times E$
are well-defined modulo relations in 
$$
\Gamma(\PP^r,\cI_E(m)) \otimes_k \Gamma(\PP^r,\cO_{\PP^r}(n)) + 
\Gamma(\PP^r,\cO_{\PP^r}(m)) \otimes_k \Gamma(\PP^r,\cI_E(n)).
$$
As the kernel of the surjective homomorphism
$$
\Gamma(\PP^r,\cO_{\PP^r}(m)) \otimes_k \Gamma(\PP^r,\cO_{\PP^r}(n))
\longrightarrow 
\Gamma(E,\cL^m) \otimes_k \Gamma(E,\cL^n)),
$$
its dimension is
$$
\binom{m+r}{r} \binom{n+r}{r} - mnd^2.
$$

In particular, this space of relations will be of interest in 
the case of minimal bidegree $(m,n) = (2,2)$ for addition laws, 
where it becomes:
$$
\binom{d+1}{2}^2 - 4d^2 = \frac{d^2(d-3)(d+5)}{4}\cdot
$$
For $d = 3$, this dimension is zero since there are no degree-2 
relations, but for $d = 4$, $5$ or $6$, the dimensions, 36, 125, 
and 297, respectively, are significant and provide a large search 
space in which to find sparse or efficiently computable forms in 
a coset.  

\subsection*{A category of pairs}

The formalization of the above concepts is provided by the 
introduction of a category of pairs $(X,\cL)$, consisting 
of a variety $X$ and very ample invertible sheaf $\cL$.  
For more general varieties $X$, in order to maintain 
the correspondence between the spaces of sections 
$\Gamma(X,\cL^n)$ and spaces of homogeneous functions of 
degree $n$ on $X$, the embedding determined by 
$\cL$ should be projectively normal.
The isomorphisms $\phi: (X_1,\cL_1) \rightarrow (X_2,\cL_2)$ in 
this category are isomorphisms $X_1 \rightarrow X_2$ for which 
$\phi^*\cL_2 \isom \cL_1$. These are the linear isomorphisms 
whose classification, for elliptic curves, is recalled in the 
next section.  
In general the space of tuples of defining polynomials of degree~$n$ 
can be identified with $\Hom(\phi^*\cL_2,\cL_1^n)$. The {\it exact} 
morphisms, for which $\phi^*\cL_2 \isom \cL_1^n$ for some $n$, are 
the subject of Section~\ref{sec:exact_morphisms}.  

\section{Linear classification of models}
\label{sec:linear_isomorphisms}

Hereafter we consider only projectively normal models. 
A linear change of variables gives a model with equivalent 
arithmetic, up to additions and multiplication by constants, 
thus it is natural to consider {\it linear isomorphisms} 
between models of elliptic curves. 
In this section we recall results from Kohel~\cite{Kohel2011} 
classifying elliptic curve morphisms which are linear. 
This provides the basis for a generalization to exact morphism 
in the next section.

\begin{definition}
Suppose that $E \subset \PP^r$ is a projectively normal 
model of an elliptic curve. 
The point $T = P_0 + \cdots P_r$, where 
$H \cap E = \{P_0,\dots,P_r\}$ for a hyperplane $H$ 
in $\PP^r$, is an invariant of the embedding called 
the embedding class of the model. The divisor 
$r(\oO)+(T)$ is called the embedding divisor class. 
\end{definition}

We recall a classification of elliptic curves models 
up to projective linear equivalence (cf.~Lemmas~2 
and~3 of Kohel~\cite{Kohel2011}).  

\begin{theorem}
Let $E_1$ and $E_2$ be two projectively normal models 
of an elliptic curve $E$ in $\PP^r$. 
There exists a linear transformation of $\PP^r$ inducing 
an isomorphism of $E_1$ to $E_2$ if and only if $E_1$ 
and $E_2$ have the same embedding divisor class. 
\end{theorem}

\noindent{\bf Remark.}
The theorem is false if the isomorphism in the 
category of elliptic curves is weakened to an 
isomorphism of curves. 
In particular, if $Q$ is a point of $E$ such that 
$[d](Q) = T_2-T_1$, then the pullback of the 
embedding divisor class $r(O)+(T_2)$ by the 
translation morphism $\tau_Q$ is $r(O)+(T_1)$,
and $\tau_Q$ is given by a projectively linear 
transformation 
(see Theorem~\ref{thm:linear-translation} for 
this statement for $T_1 = T_2$). 

\begin{corollary}
Two projectively normal models for an elliptic curve  
of the same degree have equivalent arithmetic up to 
additions and multiplication by fixed constants if 
they have the same embedding divisor class. 
\end{corollary}

A natural condition is to assume that $[-1]$ is also 
linear on $E$ in its embedding, for which we recall 
the notion of a symmetric model 
(cf.~Lemmas~2 and~4 of Kohel~\cite{Kohel2011} for  
the equivalence of the following conditions).

\begin{definition}
A projectively normal elliptic curve model $\iota: E \rightarrow \PP^r$
is {\it symmetric} if and only if any of the following is 
true:
\begin{enumerate}
\item
$[-1]$ is given by a projective linear transformation,
\item
$[-1]^*\cL \isom \cL$ where $\cL = \iota^*\cO_{\PP^r}(1)$,
\item
$T \in E[2]$, where $T$ is the embedding class.
\end{enumerate}
\end{definition}
\noindent

In view of the classification of the linear isomorphism class, 
this reduces the classification of projectively normal symmetric 
models of a given degree $d$ to the finite set of points $T$ in $E[2]$
(and more precisely, for models over $k$, to $T$ in $E[2](k)$). 

To complete the analysis of models up to linear equivalence, 
we finally recall a classification of linear translation maps.
Although the automorphism group of an elliptic curve is finite, 
and in particular $\Aut(E) = \{\pm 1\}$ if $j(E) \ne 0,12^3$,
there exist additional automorphisms as genus-one curves:
each point $T$ induces a translation-by-$T$ morphism $\tau_T$.
Those which act linearly on a given model have the following 
simple characterization (see Lemma~5 of Kohel~\cite{Kohel2011}). 

\begin{theorem}
\label{thm:linear-translation}
Let $E$ be a projectively normal projective degree $d$ model of an elliptic 
curve. The translation-by-$T$ morphism $\tau_T$ acts linearly if 
and only if $T$ is in $E[d]$.
\end{theorem}

\noindent{\bf Remark.} The statement is geometric, in the sense 
that it is true for all $T$ in $E(\bar{k})$, but if $T$ is not 
in $E(k)$ then the linear transformation is not $k$-rational.

\section{Exact morphisms and isogenies}
\label{sec:exact_morphisms}

In order to minimize the number of arithmetic operations, 
it is important to control the degree of the defining 
polynomials for an isogeny. For an isomorphism, we gave 
conditions for the isomorphism to be linear.  In general 
we want to characterize those morphisms of degree~$n$ 
given by polynomials of degree~$n$. 

A tuple $(f_0,\dots,f_r)$ of polynomials defining a morphism 
$\phi: X \rightarrow Y$ as a rational map is defined to be 
{\it complete} if the exceptional set 
$$
\{ P \in X(\bar{k}) \;|\; f_0(P) = \cdots = f_r(P) = 0 \} 
$$
is empty. In this case a single tuple defines $\phi$ as a morphism.  
The following theorem characterizes the existence and uniqueness of 
such a tuple for a morphism of curves. 

\begin{theorem}
\label{thm:exact-isogeny-sheaves}
Let $\phi: C_1 \rightarrow C_2$ be a morphism of curves, embedded 
as projectively normal models by invertible sheaves, $\cL_1$ and 
$\cL_2$, respectively.  
The morphism $\phi$ is given by a complete tuple 
$\fs = (f_0,\dots,f_r)$ of defining polynomials of degree~$n$ 
if and only if $\phi^*\cL_2 \isom \cL_1^n$. If it exists, $\fs$ 
is unique in $k[C_1]^d$ up to a scalar multiple. 
\end{theorem}

\begin{proof}
Under the hypotheses that the $C_i$ are projectively normal 
models, we identify the spaces of polynomials of degree~$n$ 
with global sections of $\cL_1^n$.  A tuple of polynomials of 
degree~$n$ defining $\phi$ corresponds to an element of 
$$
\Hom(\phi^*\cL_2,\cL_1^n) \isom
    \Gamma(C_1,\phi^*\cL_2^{-1} \otimes \cL_1^n ).
$$
Being complete implies that $\fs$ is a generator for 
all such tuples of degree~$n$ defining polynomials for 
$\phi$, as a $k = \Gamma(C_1,\cO_{C_1})$ vector space. 
Explicitly, let $(g_0,\dots,g_r)$ be another tuple, and 
set 
$
c = g_0/f_0 = \cdots = g_r/f_r \in k(C_1).
$
Since the $f_j$ have no common zero, $c$ has no poles  
and thus lies in $k$. Consequently 
$$
k = \Gamma(C_1,\phi^*\cL_2^{-1} \otimes \cL_1^n), 
\mbox{ and hence } \phi^*\cL_2 \isom \cL_1^n.  
$$
Conversely, if the latter isomorphism holds, $\Hom(\phi^*\cL_2,\cL_1^n) \isom k$,
and a generator $\fs$ for $\Hom(\phi^*\cL_2,\cL_1^n)$ is also 
a generator of the spaces of defining polynomials of all degrees:
$$
\Hom(\phi^*\cL_2,\cL_1^{n+m}) = k \fs \otimes_k \Gamma(C_1,\cL_1^m).
$$
Since $\phi$ is a morphism, $\fs$ has no base point, hence is complete.
\end{proof}

We say that a morphism between projectively normal models of curves 
is {\it exact} if it satisfies the condition $\phi^*\cL_2 \isom \cL_1^n$ 
for some $n$. If $\phi$ is exact, then $n$ is uniquely determined by 
$$
\deg(\phi)\deg(\cL_2) =
\deg(\phi^*\cL_2) = 
\deg(\cL_1^n) = n\deg(\cL_1),
$$
and, in particular $n = \deg(\phi)$ if $C_1$ and $C_2$ 
are models of the same degree. 

\begin{corollary}
\label{cor:exact-isogeny-divisors}
Let $E_1$ and $E_2$ be projecively normal models of elliptic curves 
of the same degree~$d$ with embedding classes $T_1$ and $T_2$.  
An isogeny $\phi: E_1 \rightarrow E_2$ of degree~$n$ and 
kernel $G$ is exact if and only if 
$$
n(T_1 - S_1) = d \sum_{Q \in G} Q
\mbox{ where } S_1 \in \phi^{-1}(T_2).
$$
\end{corollary}

\begin{proof}
This statement expresses the sheaf isomorphism $\cL_1^n \isom \phi^*\cL_2$ 
in terms of equivalence of divisors:
$$
n((d-1)(O_1) + (T_1)) = n D_1 \sim \phi^* D_2 = \phi^*((d-1)(O_2) + (T_2)).
$$
This equivalence holds if and only if the evaluation of the divisors on 
the curve are equal, from which the result follows.
\end{proof}

\begin{corollary}
The multiplication-by-$n$ map on any symmetric projectively normal model is exact. 
\end{corollary}

\begin{proof}
In the case of a symmetric model we take $E = E_1 = E_2$
in the previous corollary. The embedding divisor class 
$T = T_1 = T_2$ is in $E[2]$, and $S \in [n]^{-1}(T)$ 
satisfies $nS = T$, so   
$$
\deg([n])(T-S) = n^2(T - S) = n(nT - T) = n(n-1)T = \oO.
$$
On the other hand, the sum over the points of $E[n]$ is 
$\oO$, hence the result.
\end{proof}

This contrasts with the curious fact that 2-isogenies 
are not well-suited to elliptic curves in Weierstrass 
form.

\begin{corollary}
There does not exist an exact cyclic isogeny of even degree~$n$ 
between curves in Weierstrass form.
\end{corollary}

\begin{proof}
For a cyclic subgroup $G$ of even order, the sum over 
its points is a nontrivial $2$-torsion point $Q$. 
For Weierstrass models we have $T_1 = \oO_1$ and 
$T_2 = \oO_2$, and may choose $S_1 = \oO_1$, so that 
(for $d = 3$)
$
n(T_1 - S_1) = \oO \ne 3Q = Q,
$
so we never have equality.
\end{proof}

\noindent{\bf Example.}
Let $E: Y^2Z = X(X^2 + aXZ + bZ^2)$ be an elliptic curve 
with rational $2$-torsion point $(0:0:1)$. The quotient 
by $G = \langle(0:0:1)\rangle$,
to the curve $Y^2Z = X((X - aZ)^2 - 4bZ^2)$, 
is given by a 3-dimensional space of polynomial maps of degree~3:
$$
(X:Y:Z) \longmapsto 
\left\{
\begin{array}{@{\,}l}
( Y^2Z : (X^2 - bZ^2)Y : X^2Z)\\
( (X + aZ)Y^2 : (Y^2 - 2bXZ - abZ^2)Y : X^2(X + aZ) )\\
( (X^2 + aXZ + bZ^2)Y : X Y^2 - b(X^2 + aXZ + bZ^2)Z : XYZ )
\end{array}
\right.
$$
but not by any system of polynomials of degree~$2$. 

\begin{corollary}
Let $\phi: E_1 \rightarrow E_2$ be an isogeny of even 
degree~$n$ of symmetric models of elliptic curves of 
the same even degree~$d$, and let $T_1$ and $T_2$ be the 
respective embedding classes. Then $\phi$ is exact
if and only if $T_2 \in \phi(E_1[n])$. 
\end{corollary}

\begin{proof}
This is a consequence of Corollary~\ref{cor:exact-isogeny-divisors}.
Since $n$ is even and $E_1$ symmetric, $nT_1 = O_1$, and 
since $d$ is even, 
$$
d \sum_{Q \in G} Q = O_1.
$$
This conclusion follows since $nS_1 = \hat\phi\phi(S_1) 
= \hat\phi(T_2)$, which equals $\oO_1$ if and only $T_2$ 
is in $\phi(E_1[n])$. 
\end{proof}

\section{Other models for elliptic curves}
\label{sec:alternative_models}

Alternative models have been proposed for efficient arithmetic 
on elliptic curves.  Since the classification of models up to 
isomorphism is more natural for projective embeddings, providing 
a reduction to linear algebra, we describe how to interpret other 
models in terms of a standard projective embedding. 
\vspace{1mm}

\noindent{\bf Affine models.} 
An affine plane model in $\AA^2$ provides a convenient means 
of specifying (an open neighborhood of) an elliptic curve.
A direct description of arithmetic in terms of the affine 
model requires inversions, interpolation of points, and 
special conventions for representations of points at infinity,
which we seek to avoid.

Affine models of degree~$3$ extend naturally to an embedding 
in the projective closure $\PP^2$ of $\AA^2$.  
When the degree of the model is greater than three, the standard 
projective closure is singular.  However, in general there 
exists a well-defined divisor at infinity of degree~$d\,(= r + 1)$, 
which uniquely determines a Riemann--Roch space and associated 
embedding in $\PP^r$, up to linear isomorphism.

\vspace{1mm}

\noindent{\bf Product space $\PP^1 \times \PP^1$.} 
Elliptic curves models in $\PP^1 \times \PP^1$ arise 
naturally by equipping an elliptic curve $E$ with 
two independent maps to $\PP^1$. The product projective 
space $\PP^1 \times \PP^1$ embeds via the Segre embedding 
as the hypersurface $X_0 X_3 = X_1 X_2$ in $\PP^3$. 
This construction is particularly natural when the 
maps to $\PP^1$ are given by inequivalent divisors 
$D_1$ and $D_2$ of degree two (such that the coordinate 
function are identified with the Riemann--Roch basis),
in which case the Segre embedding of $\PP^1 \times \PP^1$ 
in $\PP^3$ induces an embedding by the Riemann--Roch 
space of $D = D_1 + D_2$.  In order for both $D_1$ 
and $D_2$ to be symmetric (so that $[-1]$ stabilizes
each of the projections to $\PP^1$), each must be 
of the form $D_i = (O) + (T_i)$ for points 
$T_i \in E[2]$.  Moreover, for the $D_i$ to be 
independent, $T_1 \ne T_2$, which implies that 
$D$ is not equivalent to~$4(O)$. 
\vspace{1mm}

\noindent{\bf Weighted projective spaces.}
Various embeddings of elliptic curves in weighted projective 
spaces appear in the computational and cryptographic literature 
for optimization of arithmetic on elliptic curves (particularly 
of isogenies). We detail a few of the standard models below, 
and their transformation to projective models of degree~$3$ 
or~$4$.  

\vspace{1mm}

\noindent{$\bullet~~\PP^2_{2,3,1}$.} An elliptic curve in this 
weighted projective space is referred to as being in Jacobian 
coordinates~\cite{HEHCC-AEC}, taking the Weierstrass form
$$
Y(Y + a_1XZ + a_3Z^3) = X^3 + a_2 X^2 Z^2 + a_2 X Z^4 + a_6 Z^6.
$$ 
The space encodes the order of the polar divisor of the functions 
$x$ and $y$ of a Weierstrass model.  An elliptic curve in this coordinate system embeds 
as a Weierstrass model in the ordinary projective plane $\PP^2$ by the map
by $(X:Y:Z) \mapsto (XZ:Y:Z^3)$ with birational inverse 
$(X:Y:Z) \mapsto (XZ:YZ^2:Z)$ 
defined outside of $(0:1:0)$ (whose image is $(1:1:0)$).

This weighted projective space gives interesting algorithmic efficiencies, 
since an isogeny can be expressed in the form
$$
P \longmapsto (\phi(P):\omega(P):\psi(P)) = 
  \left(\frac{\phi(P)}{\psi(P)^2}:\frac{\omega(P)}{\psi(P)^3}:1\right)\cdot
$$
Unfortunately, addition doesn't preserve the poles, so mixing 
isogenies (e.g.~doublings and triplings) with addition, one loses 
the advantages of the special form. 

\vspace{1mm}

\noindent{$\bullet~~\PP^2_{1,2,1}$.}
An elliptic curve in this weighted projective space is referred 
to as being in López--Dahab coordinates~\cite{HEHCC-AEC}. 
This provides an artifice for deflating a model in $\PP^3$ to 
the surface $\PP^2_{1,2,1}$.  
It embeds as the surface $X_0 X_2 = X_1^2$ in $\PP^3$ by 
$$
(X:Y:Z) \mapsto (X^2:XZ:Z^2:Y),
$$
with inverse 
$$
(X_0:X_1:X_2:X_3) \longmapsto 
\left\{
\begin{array}{@{\,}c}
(X_1:X_2 X_3:X_2),\\
(X_0:X_0 X_3:X_1).
\end{array}
\right.
$$
\vspace{1mm}

\noindent{$\bullet~~\PP^3_{1,2,1,2}$.}
An elliptic curve in this weighted projective space is commonly referred 
to as being in extended López--Dahab coordinates. Denoting the coordinates 
$(X:Y:Z:W)$, an elliptic curve is usually embedded in the surface $S: 
W = Z^2$ (variants have $W = XZ$ or extend further to include $XZ$ and $Z^2$).  
As above, the space $S$ is birationally equivalent to $\PP^3$:
$$
(X:Y:Z:W) \mapsto (X^2:XZ:Z^2:Y) = (X^2:XZ:W:Y).
$$
For isogenies (e.g.~doubling and tripling), by replacing a final squaring 
with an inital squaring, one can revert to $\PP^2_{1,2,1}$. 

\section{Efficient arithmetic}
\label{sec:efficient_arithmetic}

We first recall some notions of complexity, which we 
use to describe the cost of evaluating the arithmetic 
on elliptic curves. 
The notation $\Mul$ and $\Sqr$ denote the cost of 
a field multiplication and squaring, respectively. 
For a finite field of $q$ elements, typical algorithms 
for multiplication take time $c_M \log(q)^\omega$ 
for some $1 + \varepsilon \le \omega \le 2$, with a possibly 
better constant for squaring (or in characteristic 2 where 
squarings can reduce to the class $O(\log(q))$). 
The upper bound of $2$ arises by a naive implementation, 
while a standard Karatsuba algorithm gives $\omega = \log_2(3)$, 
and fast Fourier transform gives an asymptotic complexity of 
$1 + \varepsilon$.  We ignore additions, which lie in the class 
$O(\log(q))$, and distiguish multiplication by a constant (of 
fixed small size or sparse), using the notation $\mul$ for its 
complexity.

The principle focus for efficient arithmetic is the operation 
of scalar multiplication by $k$.  Using a windowing computation,
we write $k = \sum_{i=0}^t a_i n^i$ in base $n = \ell^k$
(the window), and precompute $[a_i](P)$ for $a_i$ in a set 
of coset representatives for $\ZZ/n\ZZ$.  A {\it sliding} 
window lets us restrict representatives for $a_i$ to $(\ZZ/n\ZZ)^*$.
We may then compute $[k](P)$, using at most $t$~additions 
for $kt$ scalings by $[\ell]$.

In order to break down the problem further, we suppose the 
existence of an isogeny decomposition $[\ell] = \hat\phi\phi$,
for which we need a rational cyclic subgroup $G \subset E[n]$
(where in practice $n = \ell = 3$ or $n = \ell^2 = 4$ --- 
the window may be a higher power of $\ell$). 
For this purpose we study families of elliptic curves with 
$G$-level structure.  
In view of the analysis of torsion action and degrees of defining 
polynomials, we give preference to degree-$d$ models where 
$n$ divides $d$, and $G$ will be either $\ZZ/n\ZZ$ or $\bbmu_n$ 
as a group scheme.   

We now describe the strategy for efficient isogeny computation.
Given $E_1$ and $E_2$ in $\PP^r$ with isogeny 
$\phi: E_1 \rightarrow E_2$ given by defining polynomials 
$(f_0,\dots,f_r)$ of degree~$n = \deg(\phi)$, we set 
$
V_0 = \Gamma(\PP^r,\cI_E(n)) 
    = \ker\big(\Gamma(\PP^r,\cO(n)) \rightarrow \Gamma(E_1,\cL^n)\big),
$
and successively construct a flag 
$$
V_0 \subset V_1 \subset \cdots \subset V_d = V_0 + \langle f_0,\dots,f_r \rangle
$$
such that each space $V_{i+1}$ is constructed by adjoining to $V_i$ 
a new form $g_i$ in $V_d \backslash V_i$, 
whose evaluation minimizes the number of $\Mul$ and $\Sqr$.
Subsequently the forms $f_0,\dots,f_r$ can be expressed in terms 
of the generators $g_0,\dots,g_r$ with complexity $O(\mul)$. 

In Sections~\ref{sec:cubic_models} and~\ref{sec:level2_models}
we analyze the arithmetic of tripling and doubling, on a family 
of degree~$3$ with a rational $3$-torsion point and a family of 
degree~$4$ with rational $2$-torsion point, respectively, such 
that the translation maps are linear. 
Let $G$ be the subgroup generated by this point. Using the $G$-module 
structure, and an associated norm map, we construct explicit generators 
$g_i$ for the flag decompositions.  
In Sections~\ref{sec:tripling_comparison} and~\ref{sec:doubling_comparison}
we compare the resulting algorithms of 
Sections~\ref{sec:cubic_models} and~\ref{sec:level2_models} to previous work.

\subsection{Arithmetic on cubic models}
\label{sec:cubic_models}

For optimization of arithmetic on a cubic family we consider a 
univeral curve with $\bbmu_3$ level structure, the 
{\it twisted Hessian normal form}:
$$
H : a X^3 + Y^3 + Z^3 = X Y Z,\ \oO = (0:1:-1), 
$$
obtained by descent of the Hessian model $X^3 + Y^3 + Z^3 = c X Y Z$ 
to $a = c^3$, by coordinate scaling (see~\cite{BernsteinKohelLange}).  
Addition on this model is reasonably efficient at a cost of $12\Mul$.
In order to optimize the tripling morphism $[3]$, we consider the 
quotient by $\bbmu_3 = \langle(0:\omega:-1) \rangle$.
\vspace{1mm}

By means of the isogeny $(X:Y:Z) \mapsto (aX^3:Y^3:Z^3)$, 
with kernel $\bbmu_3$, we obtain the quotient elliptic curve
$$
E : X Y Z = a (X + Y + Z)^3,\ \oO = (0:1:-1).
$$
This yields an isogeny $\phi$ of cubic models by construction, 
at a cost of three cubings: $3\Mul + 3\Sqr$.

In the previous construction, by using the $\bbmu_n$ structure, with 
respect to which the coordinate functions are diagonalized, we were 
able to construct the quotient isogeny 
$$
(X_0:\dots:X_r) \mapsto (X_0^n:\dots:X_r^n)
$$ 
without much effort.  It remains to construct the dual.  
\vspace{2mm}

In the case of the twisted Hessian, the dual isogeny $\psi = \hat\phi$ 
is given by $(X:Y:Z) \mapsto (f_0:f_1:f_2)$, where
$$
\begin{array}{l}
f_0 = X^3 + Y^3 + Z^3 - 3 X Y Z,\\
f_1 = X^2 Y + Y^2 Z + X Z^2 - 3 X Y Z,\\
f_2 = X Y^2 + Y Z^2 + X^2 Z - 3 X Y Z
\end{array}
$$
as we can compute by pushing $[3]$ through $\phi$.  

The quotient curve $E : X Y Z = a (X + Y + Z)^3$, admits 
a $\ZZ/3\ZZ$-level structure, acting by cyclic coordinate 
permutation.  The isogeny $\psi: E \rightarrow H$ is the 
quotient of this group $G = \ker(\psi)$ must be defined by 
polynomials in 
$$
\Gamma(E,\cL_E^3)^G = 
\big\langle
  X^3 + Y^3 + Z^3, X^2Y + Y^2Z + XZ^2, XY^2 + YZ^2 + X^2Z
\big\rangle
$$
modulo the relation $XYZ = a(X + Y + Z)^3$.
We note, however, that the map
$
\psi^*: \Gamma(H,\cL_H) \rightarrow \Gamma(E,\cL_E^3)^G
$
must be surjective since both have dimension $3$. 
\vspace{1mm}

Using the group action, we construct the norm map
$$
N_G : \Gamma(E,\cL) \rightarrow \Gamma(E,\cL_E^3)^G,
$$ 
by $N_G(f) = f(X,Y,Z)f(Y,Z,X)f(Z,X,Y)$. It is nonlinear 
but sufficient to provide a set of generators using 
$2\Mul$ each, and by fixing a generator of the fixed 
subspace of $G$, we construct a distinguished generator 
$g_0$ requiring $1\Mul + 1\Sqr$ for cubing. 

For the first norm we set $g_0 = N_G(X+Y+Z) = (X+Y+Z)^3$, 
noting that $N_G(X) = N_G(Y) + N_G(Z) = X Y Z = a g_0$.  
We complete a basis with forms $g_1$ and $g_2$ given by 
$$
\begin{array}{r@{\;}c@{\;}l}
g_1 = N_G(Y+Z) & = & (Y+Z)(X+Z)(X+Y), \\
g_2 = N_G(Y-Z) & = & (Y-Z)(Z-X)(X-Y),
\end{array}
$$
then solve for the linear transformation to the basis 
$\{f_0,f_1,f_2\}$:
$$
\begin{array}{r@{\;}c@{\;}l}
f_0 & = & (1-3a)g_0 - 3g_1,\\
f_1 & = & -4ag_0 + (g_1-g_2)/2,\\
f_2 & = & -4ag_0 + (g_1+g_2)/2.
\end{array}
$$
This gives an algorithm for $\psi$ using $5\Mul + 1\Sqr$, for 
a total tripling complexity of $8\Mul + 4\Sqr$ using the 
decomposition $[3] = \psi\circ\phi$. 
Attributing $1\mul$ for the multiplications by~$a$, ignoring 
additions implicit in the small integers (after scaling by $2$), 
this gives $8\Mul + 4\Sqr + 2\mul$. 

\subsection{Comparison with previous work}
\label{sec:tripling_comparison}

A naive analysis, and the previously best known algorithmm 
for tripling, required $8\Mul + 6\Sqr + 1\mul$.  To compare 
with scalar multiplication using doubling and a binary chain, 
one scales by $\log_3(2)$ to account for the reduced length 
of the addition chain. 

For comparison, the best known doublings algorithms on ordinary 
projective models (in characteristic other than~2) are:
\begin{itemize}
\item
Extended Edwards models in $\PP^3$, using $4\Mul + 4\Sqr$.
(Hisil et al.~\cite{Hisil-EdwardsRevisited})
\item
Singular Edwards models in $\PP^2$, using $3\Mul + 4\Sqr$ 
(Berstein et al.~\cite{BernsteinEtAl-TwistedEdwards})
\item
Jacobi quartic models in $\PP^3$, using $2\Mul + 5\Sqr$.
(Hisil et al.~\cite{Hisil-EdwardsRevisited})
\end{itemize}
We note that the Jacobi quartic models are embeddings in $\PP^3$ 
of the affine curve $y^2 = x^4 + 2a x^2 + 1$ extended to a 
projective curve in the $(1,2,1)$-weighted projective plane. 
The embedding $(x,y) \mapsto (x^2:y:1:x) = (X_0:X_1:X_2:X_3)$ 
gives
$$
X_1^2 = X_0^2 + 2a X_0 X_2 + X_2^2,\ X_0 X_2 = X_3^2, 
$$
in ordinary projective space $\PP^3$. There also exist models 
in weighted projective space with complexity $2\Mul + 5\Sqr$ on 
the 2-isogeny oriented curves~\cite{DIK-2006} with improvements 
of Bernstein and Lange~\cite{BernsteinLange-Edwards}, and a 
tripling algorithm with complexity of $6\Mul + 6\Sqr$ for 
3-isogeny oriented curves~\cite{DIK-2006}.
Each of these comes with a significantly higher cost for addition
(see~\cite{DIK-2006}, the EFD~\cite{EFD}, and the table below 
for more details).

The relative comparison of complexities of $[\ell]$ and addition 
$\oplus$ on twisted Hessians ($[\ell] = [3]$) and on twisted 
Edwards models and Jacobi quartics ($[\ell] = [2]$) yields the 
following:
$$
\begin{array}{c@{}lccc|l}
   & & \multicolumn{3}{c}{\mbox{ Cost of $1\Sqr$ }}\\
\;[\ell]\;     & & 1.00\Mul & 0.80\Mul & 0.66\Mul  & \mbox{\hspace{1mm}}\oplus \\ \hline
4\Mul + 4\Sqr  & & 8.00\Mul & 7.20\Mul & 6.66\Mul & 9\Mul \\
(8\Mul + 4\Sqr)&
     \log_3(2)   & 7.57\Mul & 7.07\Mul & 6.73\Mul & 12\Mul\log_3(2) = 7.57\Mul\\
(6\Mul + 6\Sqr)&
     \log_3(2)   & 7.57\Mul & 6.81\Mul & 6.28\Mul & (11\Mul + 6\Sqr)\log_3(2)\\
3\Mul + 4\Sqr  & & 7.00\Mul  & 6.20\Mul & 5.66\Mul & 10\Mul + 1\Sqr\\
2\Mul + 5\Sqr  & & 7.00\Mul  & 6.00\Mul & 5.33\Mul & 7\Mul + 3\Sqr
\end{array}
$$
This analysis brings tripling on a standard projective model, 
coupled with an efficient addition algorithm, in line with 
with doubling (on optimal models for each). 
In the section which follows we improve the $2\Mul + 5\Sqr$ 
result for doubling. 

\subsection{Arithmetic on level-$2$ quartic models}
\label{sec:level2_models}

The arithmetic of quartic models provides the greatest advantages 
in terms of existence exact $2$-isogeny decompositions and symmetric 
action of $4$-torsion subgroups. 
A study of standard models with a level-$4$ structure, which provide 
the best complexity for addition to complement doubling complexities, 
will be detailed elsewhere. 
The best doubling algorithms, however, are obtained for embedding 
divisor class $4(O)$, as in the Jacobi quartic, rather than 
$3(O) + (T)$, for a $2$-torsion point $T$, as is the case for the 
Edwards model (see~\cite{Edwards} and \cite{BernsteinLange-Edwards})
or its twists, the $\ZZ/4\ZZ$-normal form or the $\bbmu_4$-normal form 
in characteristic $0$.  

In what follows we seek the best possible complexity for doubling in
a family of elliptic curves.  
In order to exploit an isogeny decomposition for doubling and linear 
action of torsion, we construct a universal family with $2$-torsion 
point and embed the family in $\PP^3$ by the divisor $4(O)$. 
We note that any of the recent profusion of models with rational 
$2$-torsion point can be transformed to this model, hence the 
complexity results obtained apply to any such family. 

\ignore{
The $2$-isogeny oriented curves of Doche, Icart, and Kohel~\cite{DIK-2006}: 
$$
y^2 = x^3 + u x^2 + 16 u x,
$$
do not cover all twists, and are not evidently suitable for 
characteristic~2.  Here we treat a more general family, and, using 
an embedding in $\PP^3$, improve on the previously best known 
doubling complexity.}

\subsubsection*{A universal level-$2$ curve.}
Over a field of characteristic different from~$2$, a general Weierstrass 
model with $2$-torsion point has the form $y^2 = x^3 + a_1 x^2 + b_1 x$.
The quotient by the subgroup $\langle (0,0)\rangle$ of order $2$ gives
$
y^2 = x^3 + a_2 x^2 + b_2 x,
$
where $a_2 = -2a_1$ and $b_2 = a_1^2 - 4b_1$, by formulas of Vélu~\cite{Velu}. 
In order to have a family with good reduction at $2$, we may express 
$(a_1,b_1)$ by the change of variables $a_1 = 4u + 1$ and $b_1 = -16v$, 
after which $y^2 = x^3 + a_1 x^2 + b_1 x$ is isomorphic to the curve
$$
E_1: y^2 + xy = x^3 + u x^2 - v x
$$
with isogenous curve 
$
E_2: y^2 + xy = x^3 + u x^2 + 4v x + (4u + 1)v.
$
A quartic model in $\PP^3$ for each of these curves is given by the 
embedding
$$
(x,y) \longmapsto (X_0,X_1,X_2,X_3) = (x^2, x, 1, y),
$$
with respect to the embedding divisor $4(O)$. This gives the quartic 
curve in $\PP^2$ given by
$$
Q_1: X_3^2 + X_1 X_3 = (X_0 + u X_1 - v X_2) X_1,\ X_1^2 = X_0 X_2,
$$
with isogenous curve
$$
Q_2 : X_3^2 + X_1 X_3 = (X_0 + 4vX_2)(X_1 + uX_2)+ v X_2^2,\ X_1^2 = X_0 X_2
$$
each having $(1:0:0:0)$ as identity. The Weierstrass model has discriminant 
$\Delta = v^2 ((4u + 1)^2 - 64v)$, hence the $E_i$ and $Q_i$ are elliptic 
curves provided that $\Delta$ is nonzero. 

Translating the Vélu $2$-isogeny through to these models we find the 
following expressions for the isogeny decomposition of doubling.

\begin{lemma}
The $2$-isogeny $\psi: Q_1 \rightarrow Q_2$ with kernel $\langle (0:0:1:0) \rangle$
sends $(X_0,X_1,X_2,X_3)$ to
$$
\big(\,(X_0 - vX_2)^2,\; (X_0 - vX_2)X_1,\; X_1^2,\; vX_1 X_2 + (X_0 + v X_2)X_3\,\big)
$$
and the dual isogeny $\phi$ sends $(X_0,X_1,X_2,X_3)$ to
$$
\big(\,(X_0 + 4vX_2)^2,\; (X_1 + 2X_3)^2,\; (4X_1 + (4u + 1)X_2)^2,\; f_3\,\big),
$$
where
$
f_3 = u X_1^2 - 8v X_1X_2 - (4u + 1)vX_2^2 + 2X_0X_3 + 4u X_1 X_3 - 8vX_2X_3 - X_3^2).
$
\end{lemma}

\subsubsection*{Efficient isogeny evaluation}
For each of the tuples $(f_0,f_1,f_2,f_3)$, we next determine 
quadratic forms $g_0$, $g_1$, $g_2$, $g_3$, each a square or 
product, spanning the same space and such that the basis 
transformation involves only coefficients which are polynomials 
in the parameters $u$ and $v$. 
In order to determine a projective isomorphism, it is necessary 
and sufficient that the determinant of the transformation be 
invertible, but it is not necessary to compute its inverse. 
As previously noted, the evaluation of equality among quadratic 
polynomials on the domain curve $Q_i$ is in $k[Q_i]$, i.e.~modulo 
the 2-dimension space of relations for $Q_i$.

\begin{lemma}
If $k$ is a field of characteristic different from 2, the quadratic 
defining polynomials for $\psi$ are spanned by the following forms
$$
(g_0,g_1,g_2,g_3) = \big(\,
  (X_0 - vX_2)^2,\;
  (X_0 - vX_2 + X_1)^2,\;
  X_1^2,\;
  (X_0 + X_1 + vX_2 + 2X_3)^2\,\big).
$$
\end{lemma}

\begin{proof}
By scaling the defining polynomials $(f_0,f_1,f_2,f_3)$ by $4$, 
the projective transformation from $(g_0,g_1,g_2,g_3)$ is given 
by $(4f_0,4f_2) = (4g_0,4g_2)$,
$$
4f_1 = -2(f_0 - f_1 + f_2) \mbox{ and } 
4f_3 = 2f_0 - 3f_1 + 2(1 - 2(u + v)) + f_3.
$$ 
Since the transformation has determinant $32$, it defines an isomorphism
over any field of characteristic different from 2.
\end{proof}

\begin{lemma}
If $k$ is a field of characteristic 2, the quadratic defining polynomials 
for $\psi$ are spanned by the following forms
$$
(g_0,g_1,g_2,g_3) = 
\big(\,(X_0 + vX_2)^2,\; (X_1 + X_3)X_3,\; X_1^2,\; (X_0 + v(X_1 + X_3))(X_2 + X_3)\,\big).
$$
\end{lemma}

\begin{proof}
The transformation from $(g_0,g_1,g_2,g_3)$ to the tuple 
$(f_0,f_1,f_2,f_3)$ of defining polynomials is given by 
$(g_0,g_2) = (f_0,f_2)$, 
$$
(g_1,g_3) = (f_1 + uf_2, vf_1 + f_2 + f_3).
$$
The transformation has determinant 1 hence is an isomorphism.
\end{proof}

\begin{corollary}
The isogeny $\psi$ can be evaluated with $4\Sqr$ in characteristic 
different from $2$ and $2\Mul + 2\Sqr$ in characteristic $2$.
\end{corollary}

\begin{lemma}
Over any field $k$, the quadratic defining polynomials for $\phi$ are 
spanned by the square forms $(g_0,g_1,g_2,g_3)$:
$$
\big(\,
  (X_0 + 4vX_2)^2,\; (X_1 + 2X_3)^2,\;
  (4X_1 + (4u + 1)X_2)^2,\; (X_0 + (2u + 1)X_1 - 4v X_2 + X_3)^2\,\big).
$$
\end{lemma}

\begin{proof}
The forms $(g_0,g_1,g_2)$ equal $(f_0,f_1,f_2)$, and it suffices 
to verify the equality $f_3 = - g_0 - (u + 1)g_1 + vg_2 + g_3$, 
a transformation of determinant~$1$. 
\end{proof}

\begin{lemma}
If $k$ is a field of characteristic 2, the quadratic defining polynomials 
for $\phi$ are spanned by $(X_0^2, X_1^2, X_2^2, X_3^2 )$.
\end{lemma}

\begin{proof}
It is verified by inspection that the isogeny $\phi$ is defined 
by a linear combination of the squares of $(X_0,X_1,X_2,X_3)$ or 
by specializing the previous lemma to characteristic~2. 
\end{proof}

\begin{corollary}
The isogeny $\phi$ can be evaluated with $4\Sqr$ over any field.
\end{corollary}

\begin{corollary}
\label{cor:naive_level2_doubling}
Doubling on $Q_1$ or $Q_2$ can be carried out with $8\Sqr$ over 
a field of characteristic different from $2$, or $2\Mul + 6\Sqr$ 
over a field of characteristic~$2$. 
\end{corollary}

\subsubsection*{Factorization through singular quotients}
With the given strategy of computing the isogenies of projectively 
normal models $\psi: Q_1 \rightarrow Q_2$ then $\phi: Q_2 \rightarrow Q_1$, 
this result is optimal or nearly so --- to span the spaces of 
forms of dimension~$4$, in each direction, one needs at least 
four operations.  
We thus focus on replacing $Q_1$ by a singular quartic curve 
$D_1$ in $\PP^2$ such that the morphisms induced by the isogenies 
between $Q_2$ and $Q_1$ remain well-defined but for which we can 
save one operation in the construction of the coordinate functions 
of the singular curves. 
We treat characteristic different from~$2$ and the derivation of 
a doubling algorithm improving on $2\Mul + 5\Sqr$; an analogous 
construction in characteristic~2 appears in Kohel~\cite{Kohel2012}.

Let $T = (0:0:1:0)$ be the $2$-torsion point on $Q_1$, 
which acts by translation as:
$$
\tau_T(X_0:X_1:X_2:X_3) = (vX_2: -X_1: v^{-1}X_0: X_1 + X_3)
$$
Similarly, the inverse morphism is:
$$
[-1](X_0:X_1:X_2:X_3) = (X_0 : X_1 : X_2 : -(X_1 + X_3))
$$
Over a field of characteristic different from $2$, the 
morphism from $Q_1$ to $\PP^2$ 
$$
(X_0:X_1:X_2:X_3) \longmapsto (X:Y:Z) = (X_0:X_1+2X_3:X_2)
$$ 
has image curve:
$$
D_1 : (Y^2 - (4u + 1)XZ)^2 = 16XZ(X - vZ)^2,
$$
on which $\tau_T$ and $[-1]$ induce linear transformations, 
since the subspace generated by $\{ X_0, X_1 + 2X_3, X_2\}$ is 
stabilized by pullbacks of both $[-1]$ and $[\tau_T]$.  
The singular subscheme of $D_1$ is $X = vZ$, $Y^2 = (4u + 1)vZ^2$, 
which has no rational points if $(4u+1)v$ is not a square, and 
in this case, the projection to $D_1$ induces an isomorphism of 
the set of nonsingular points.  Since $\tau_T$ acts linearly, 
the morphism $\psi: Q_1 \rightarrow Q_2$ maps through $D_1$, as 
given by the next lemma.

\begin{lemma}
\label{lem:psi_square_forms}
The $2$-isogeny $\psi: Q_1 \rightarrow Q_2$ induces a morphism 
$D_1 \rightarrow Q_2$ sending $(X,Y,Z)$ to
$$
\big(\,8(X - vZ)^2,\; 2(Y^2 - (4u + 1)XZ),\; 8XZ,\; 
4Y(X + vZ) - (Y^2 - (4u + 1)XZ\, \big).
$$
This defining polynomials are spanned by
$$
(g_0,g_1,g_2) = \big(\, (X - vZ)^2,\; Y^2,\; (X + vZ)^2,\; (X + Y + vZ)^2\,\big).
$$
In particular the morphism can be evaluated with $4\Sqr$.
\end{lemma}

\begin{lemma}
\label{lem:phi_square_forms}
The $2$-isogeny $\phi: Q_2 \rightarrow Q_1$ induces a morphism 
$\phi: Q_2 \rightarrow D_1$ sending $(X_0,X_1,X_2,X_3)$ to 
$$
\big(\,
  (X_0 + 4vX_2)^2,\;
  (X_1 + 2X_3)(2X_0 + (4u + 1)X_1 - 8vX_2),\;
  (4X_1 + (4u + 1)X_2)^2\,\big),
$$
which can be evaluated with $1\Mul + 2\Sqr$.  If $4u + 1 = -(2s + 1)^2$,
then the forms 
$$
\big(\,
  (X_0 + 4vX_2)^2,\;
  (X_0 - 4vX_2 - (2s + 1)(sX_1 - X_3))^2,\;
  (4X_1 + (4u + 1)X_2)^2\,\big),
$$
span the defining polynomials for $\phi: Q_2 \rightarrow D_1$, and 
can be evaluated with $3\Sqr$.
\end{lemma}

\begin{proof}
The form of the defining polynomals $(f_0,f_1,f_2)$ for the map 
$\phi: Q_2 \rightarrow D_1$ follows from composing the $2$-isogeny 
$\phi: Q_2 \rightarrow Q_2$ with the projection to $D_1$. 
The latter statement holds since, the square forms $(g_0,g_1,g_2)$
of Lemma~\ref{lem:phi_square_forms} satisfy $f_0 = g_0$, $f_2 = g_2$, 
and $(2s + 1)f_1 = -2(g_0 -  g_1 - vg_2)$.  
\end{proof}

Composing the morphism $Q_2 \rightarrow D_1$ with $D_1 \rightarrow Q_2$ 
gives the following complexity result.

\begin{theorem}
\label{thm:optimal_level2_doubling}
The doubling map on $Q_2$ over a field of characteristic different 
from $2$ can be evaluated with $1\Mul + 6\Sqr$, 
and if $4u + 1 = -(2s + 1)^2$, with $7\Sqr$. 
\end{theorem}

\noindent
{\bf Remark.} The condition $a_1 = 4u + 1 = -(2s + 1)^2$ is 
equivalent to the condition $u = -(s^2 + s + 1/2)$. This implies 
that the curves in the family are isomorphic to one of the 
form $y^2 = x^3 - x^2 + b_1x$, where $b_1 = - 16v/(4u + 1)^2$, 
fixing the quadratic twist but not changing the level structure. 
In light of this normalization in the subfamily, we may as well
fix $s = 0$ and $4u + 1 = -1$ to achieve a simplification of the 
formulas in terms of the constants.  

\subsection{Comparison with previous doubling algorithms}
\label{sec:doubling_comparison}

We recall that the previously best known algorithms for 
doubling require $2\Mul + 5\Sqr$, obtained for Jacobi quartic 
models in $\PP^3$ (see Hisil et al.~\cite{Hisil-EdwardsRevisited})
or specialized models in weighted projective space~\cite{DIK-2006}.
We compare this base complexity to the above complexities 
which apply to any elliptic curve with a rational $2$-torsion 
point.  We include the naive $8\Mul$ algorithm of 
Corollary~\ref{cor:naive_level2_doubling}, and improvements of 
Theorem~\ref{thm:optimal_level2_doubling} to $1\Mul + 6\Sqr$ 
generically, and $7\Sqr$ for an optimal choice of twist. 
The relative costs of the various doubling algorithms, summarized 
below, show that the proposed doubling algorithms determined here 
give a non-neglible improvement on previous algorithms.  
$$
\begin{array}{rccc}
   & \multicolumn{3}{c}{\mbox{ Cost of $1\Sqr$ }}\\
\mbox{\hspace{1mm}}
\;[2]\;        & 1.00\Mul & 0.80\Mul & 0.66\Mul \\ \hline
        8\Sqr  & 8.0\Mul  & 6.40\Mul & 5.33\Mul \\
2\Mul + 5\Sqr  & 7.0\Mul  & 6.00\Mul & 5.33\Mul \\
1\Mul + 6\Sqr  & 7.0\Mul  & 5.80\Mul & 5.00\Mul \\
        7\Sqr  & 7.0\Mul  & 5.60\Mul & 4.66\Mul
\end{array}
$$
The improvements for doubling require only a $2$-torsion point, 
but imposing additional $2$-torsion or $4$-torsion structure 
would allow us to carry this doubling algorithm over to a normal 
form with symmetries admitting more efficient addition laws.

\end{document}